\documentclass[twoside,a4paper]{amsart}
\usepackage{amssymb, amsmath,latexsym, amsthm}
\usepackage{fancyhdr}

\newtheoremstyle{dotless}{}{}{\itshape}{}{\bfseries}{}{ }{}
\theoremstyle{dotless}
\newtheorem{proposition}{Proposition}[section]
\newtheorem{theorem}{Theorem}
\newtheorem{theoremm}{Theorem}
\newtheorem{lemma}{Lemma}
\theoremstyle{definition}
\DeclareMathOperator{\supp}{supp}

\DeclareMathOperator{\rank}{rank}
\DeclareMathOperator{\dist}{dist}
\DeclareMathOperator{\Lip}{Lip}

\title{Boundary oscillations of harmonic functions in Lipschitz domains}
\author{P. Mozolyako}
\address{Department of Mathematical Sciences, Norwegian University of Science andechnology, NO-7491, Trondheim, Norway}
\email{pavel.mozolyako@math.ntnu.no}
\thanks{ Author was supported by Research Council of Norway, grant 204726/V30.}
\begin{document}
\begin{abstract}
Let $u(x,y)$ be a harmonic function in the halfspace $\mathbb{R}^n\times\mathbb{R}_+$ that grows near the boundary not faster than some fixed majorant $w(y)$. Recently it was proven that an appropriate weighted average along the vertical lines of such a function satisfies the Law of Iterated Logarithm (LIL). We extend this result to a class of Lipschitz domains in $\mathbb{R}^{n+1}$. In particular, we obtain the local version of this LIL for the upper halfspace. The proof is based on approximation of the weighted averages by a Bloch function, satisfying some additional condition determined by the weight $w$. The growth rate of such Bloch function depends on $w$ and, for slowly increasing $w$, turns out to be slower than the one provided by LILs of Makarov and Llorente. We discuss the necessary condition for an arbitrary Bloch function to exhibit this type of behaviour. 
\end{abstract}
\maketitle

\renewcommand{\thetheoremm}{\Alph{theoremm}}
\section{Introduction}
Let $w:\mathbb{R}_+\rightarrow\mathbb{R}_+$ be a continuous decreasing function, 
\begin{equation}\label{e:0}
\lim_{y\rightarrow0+}w(y) = \infty,\; w(y) = 1,\, y>1,
\end{equation}
that satisfies the doubling condition
\begin{equation}\label{e:wdc}
w(y) \leq Dw(2y),\quad y\in\mathbb{R}_+,
\end{equation}
for some constant $D>0$.
Given a Lipschitz function $\phi:\mathbb{R}^n\rightarrow \mathbb{R}$, we denote by $\Omega_{\phi}$ the domain above the graph of $\phi$,
\begin{equation}\notag
\Omega_{\phi} = \{(x,y): x\in\mathbb{R}^n, y\geq \phi(x)\}.
\end{equation}
We consider harmonic functions in $\Omega_{\phi}$ with the growth restriction
\begin{equation}\notag
|u(x,y)| \leq Kw(\dist((x,y),\partial\Omega_{\phi})),\quad (x,y) \in \Omega_{\phi}.
\end{equation}
The space of these functions is denoted by $h^{\infty}_{w}(\Omega_{\phi})$ and the smallest $K$ for which this inequality is satisfied is called the norm of $u$ in $h^{\infty}_{w}(\Omega_{\phi})$. We denote it by $\|u\|_{w,\infty}$.\par
In \cite{EML:13} the following result was obtained
\begin{theoremm}
Let $u$ be a function in $h^{\infty}_w(\mathbb{R}^{n+1}_+)$. For $x\in\mathbb{R}^n$ and $0<\delta\leq1$ put
\begin{equation}\notag
I_0(x,\delta) = \int_{\delta}^1u(x,y)\,d\left(\frac{1}{w(y)}\right).
\end{equation}
Then the following inequality holds
\begin{equation}\label{e:inot}
\limsup_{\delta\rightarrow0}\frac{I_0(x,\delta)}{\sqrt{\log w(\delta)\log\log\log w(\delta)}} \leq C\|u\|_{w,\infty},\quad a.e.\, x\in\mathbb{R}^n,
\end{equation}
where the constant $C = C(n)$ depends only on the dimension $n$.
\end{theoremm}
Our main goal is to extend this result to the domains in $\mathbb{R}^{n+1}_+$ above the graphs of Lipschitz functions.
\begin{theorem}\label{t:th1}
Let $\phi$ be a Lipschitz function on $\mathbb{R}^n$ and let $u$ be a function in $h^{\infty}_w(\Omega_{\phi})$. For $x\in\mathbb{R}^n$ and $0<\delta\leq1$ put 
\begin{equation}\label{e:idd}
I(x,\delta) = \int_{\delta}^1u(x,\phi(x)+y)\,d\left(\frac{1}{w(y)}\right).
\end{equation}
Then the following LIL holds
\begin{equation}\label{e:alil}
\limsup_{\delta\rightarrow0}\frac{I(x,\delta)}{\sqrt{\log w(\delta)\log\log\log w(\delta)}} \leq C\|u\|_{w,\infty},\quad a.e.\, x\in\mathbb{R}^n,
\end{equation}
where $C$ depends only on the function $\phi$, weight $w$ and dimension $n$.
\end{theorem}
We see that the condition $u \in h^{\infty}_w(\mathbb{R}^{n+1}_+)$ implies immediately that $I_0(x,\delta) \leq C\log w(\delta)$. In addition, the weighted average $I_0$ enjoys some nice cancellation properties, so, with the help of the Law of Iterated Logarithm (LIL) techniques, a better estimate \eqref{e:inot} was obtained. In doing this we relied heavily on the wavelet decomposition of $u(\cdot,\delta)$. Unfortunately these methods work only on nice domains (such as the disc or the upper halfspace), and since we need to extend this result to domains  $\Omega_{\phi}$, we have to use a different approach. This approach is based on the ideas by J. Llorente and A. Nicolau, and in the simplest case of the upper halfplane it goes as follows. First we approximate $I_0$ by a Bloch function $H$ that also happens to belong to $h^{\infty}_{\log w}(\mathbb{R}^2_+)$. Recall that $H$ is a Bloch function in $\mathbb{R}^2_+$ if it is harmonic there and $|\nabla H|(x,y) \leq \frac{C}{y}$ for $(x,y)\in \mathbb{R}^2_+$. Here we would like to mention two of the LILs for harmonic functions, namely the Makarov-Llorente LIL for the Bloch functions, \cite[Corollary \textbf{3.2}]{Mak} \cite[Theorem \textbf{1}]{Ll:98}, and the LIL of Ba\~{n}uelos-Moore, \cite[Theorem \textbf{3.04}]{BM:98}. Unfortunately, we could not use either of those directly, since the former does not provide the desired estimate for slow growing weights $w$, and the latter involves the Lusin area integral, which, we believe, can not be properly estimated by the weight. Therefore we modify the ideas used in the proof of those LILs, so we proceed by approximating $H$ by a (super)dyadic martingale and estimating its quadratic function by $w$. Then it remains to apply the LIL for the martingales.\\ We would also like to note that Theorem \ref{t:th1} remains true if we replace $\Omega_{\phi}$ with some star-like Lipschitz domain.\par
As a corollary of Theorem \ref{t:th1} we have the local version of Theorem \textbf{5} from \cite{EML:13}
\begin{theorem}\label{t:th2}
Let $u$ be a harmonic function in $\mathbb{R}^{n+1}_+$. Assume that there exists a set $\Sigma\subset\mathbb{R}^n$ of positive $n$-dimensional Lebesgue measure such that for every $x_0\in \Sigma$
\begin{equation}\label{e:cgr}
|u(x,y)| \leq Kw(y),\quad |x-x_0|\leq My,
\end{equation}
where $M$ is some positive constant. Then
\begin{equation}\label{e:th2st}
\limsup_{\delta\rightarrow0}\frac{I_0(x,\delta)}{\sqrt{\log w(\delta)\log\log\log w(\delta)}} \leq CK,\quad a.e.\, x\in\Sigma,
\end{equation}
where the constant $C$ depends only on $M$, $w$ and $n$.
\end{theorem}
Note that the condition \eqref{e:cgr} restricts the \textit{non-tangential} growth of $u$ near the boundary. We do not know if this result remains true if we replace \eqref{e:cgr} by a \textit{radial} growth condition.\par
The plan of the paper is as follows: in Section \ref{s:not} we collect some notation and known results about dyadic martingales, and in Sections \ref{ss:ml}-\ref{ss:dma} we prove Theorem \ref{t:th1}. The proof itself consists of three parts. In Section \ref{ss:hir} we approximate $I$ by a Bloch function $H$ that belongs to $h^{\infty}_{\log w}(\Omega_{\phi})$. Then, in Section \ref{ss:dma}, we approximate $H$ by a superdyadic martingale. After that we can apply the LIL for martingales to finish the proof (Section \ref{ss:mtprf}). Theorem \ref{t:th2} is proven in Section \ref{ss:crlr}. In Section \ref{s:ex} we consider the question  whether \textit{every} Bloch function in $h^{\infty}_{\log w}(\Omega_{\phi})$ satisfies the LIL \eqref{e:alil}. Using the technique of dyadic martingales we construct an example of a Bloch function that provides the negative answer to this question.
\section{Notation}\label{s:not}
By $\lambda_n$ we denote the Lebesgue measure in $\mathbb{R}^n,\; n\in\mathbb{N}$.
Given a point $x = (x_1,x_2,\dots,x_n)\in \mathbb{R}^n$ and $r>0$ we denote by $Q(x,r)$ the cube of radius $r$ centered at $x$
\begin{equation}\notag
Q(x,r) = \prod_{i=1}^n(x_i-r,x_i+r],
\end{equation}
we also put $Q(x,\frac12) := Q(x)$. Further, given a cube $Q$ we denote its center by $x_Q$, so that $Q = Q(x_Q,r)$ for some positive $r$.\par Fix $x\in\mathbb{R}^n$. If $2^kx_i-\frac12\in\mathbb{Z}$ for every $i=1\dots n$, and $r = 2^{-k-1}$ for some $k\in \mathbb{Z}_+$, we call the cube $Q = Q(x,r)$ dyadic of rank $k$. For $x\in \mathbb{R}^n$ and $k\in\mathbb{Z}_+$ we denote by $\Delta_k(x)$ the collection of all (shifted) dyadic cubes of rank $k$ in $Q(x)$,
\begin{multline*}
\Delta_k(x)=\\ = \{\prod_{i=1}^n(x-\frac12+m_i2^{-k},x-\frac12+(m_i+1)2^{-k}),\; m_i\in\mathbb{Z}_+,\; 0\leq m_i \leq 2^{k}-1\},
\end{multline*}
we also put $\Delta(x) = \bigcup_{k=0}^{\infty}\Delta_k(x)$. If $Q(x) = (0,1]^n$, we write $\Delta_k$ and $\Delta$ respectively. By $\mathcal{F}_k(x)$ we denote the sigma-algebra generated by dyadic cubes of rank $k$ in $Q(x)$, 
\begin{equation}\notag
\mathcal{F}_k(x) = \sigma(\Delta_k).
\end{equation}
Given a probability Borel measure $\mu$ on $Q(x)$ and an increasing sequence $\{\alpha_k\}_{k=0}^{\infty}\subset \mathbb{Z}_+$ we can consider the (super)dyadic martingales on $Q(x)$ with respect to the filtration $\{\mathcal{F}_{\alpha_k}(x)\}_{k=0}^{\infty}$, they are usually denoted by $\Lambda = \{\Lambda_{k}, \mathcal{F}_{\alpha_k}(x),\mu\}$. This means that $\Lambda_{k}$ is a piecewise constant function on the (shifted) dyadic cubes of rank $\alpha_k$, and if $q$ is a dyadic cube in $\Delta_{\alpha_{k-1}}(x)$, then
\begin{equation}\notag
\frac{1}{\mu(q)}\int_{q}\Lambda_{k}(t)\,d\mu(t) = \Lambda_{k-1}(x_q).
\end{equation}
In particular, if $n=1$ and $\mu = \lambda_1$ is the Lebesgue measure on $Q\left(\frac{1}{2}\right) = (0,1]$ then $\Lambda$ has a following truncated wavelet representation
\begin{equation}\label{e:hpwr}
\Lambda_{k}(t) = L + \sum_{j=0}^{\alpha_k}\sum_{i=0}^{2^j-1}b_{ij}\psi_{ij}(t),\quad t\in (0,1],
\end{equation}
where $L = \mathbb{E}\Lambda_{k} = \int_{Q_0}\Lambda_k(t)\,d\lambda_1(t)$, $\psi_{ij}(t) = \psi(2^{j}t-i),\; t\in\mathbb{R}$ , and $\psi$ is the Haar wavelet, $\psi = \chi_{[0,1]} -2\chi_{[0,\frac12]}$ (instead of the usual $L^2$ scaling we use $L^{\infty}$ one here, it is more convenient for our purposes). For any interval $I\subset (0,1]$ of length $2r$, $I = [x_I-r,x_I+r]$, we put
\begin{equation}\notag
\psi_I(t) = \psi\left(\frac{t-x_I+r}{2r}\right),\quad t\in\mathbb{R}.
\end{equation}
Then \eqref{e:hpwr} can be written as follows
\begin{equation}\notag
\Lambda_{k}(t) = L + \sum_{j=0}^{\alpha_k}\sum_{I\in\Delta_{j}}b_I\psi_{I}(t),\quad t\in (0,1],
\end{equation}
where $b_I = \frac{1}{\lambda_1(I)}\int_{\mathbb{R}}\Lambda_k(t)\psi_I(t)$.\par
By $\langle\Lambda\rangle_k$ we denote the quadratic function of $\Lambda$,
\begin{equation}\notag
\langle\Lambda\rangle_k^2 = \sum_{j=1}^{k}\mathbb{E}[|\Lambda_{j}-\Lambda_{j-1}|^2\vert\mathcal{F}_{\alpha_{j-1}}].
\end{equation}
If $\alpha_k = k$ and we use the Haar representation of $\Lambda$, we can write the quadratic function in the following way
\begin{equation}\label{e:hmqf}
\langle\Lambda\rangle_k^2(t) = L^2 + \sum_{I\in\Delta:t\in I, \lambda_1(I)\geq 2^{-k}}b_I^2.
\end{equation}
Let $u$ be a harmonic function in $\Omega_{\phi}$. We say that $u$ belongs to the Bloch class in $\Omega_{\phi}$, if there exists a constant $D>0$  such that
\begin{equation}\notag
|\nabla u(x,y)| \leq \frac{D}{\dist((x,y),\partial \Omega_{\phi})},\quad (x,y) \in \Omega_{\phi}.
\end{equation}
We denote the space of such functions by $\mathcal{B}(\Omega_{\phi})$  and the smallest $D$ for which this inequality is satisfied by $\|u\|_{\mathcal{B}}$.\\
The connection between Bloch functions and dyadic martingales is well established, see, for example, \cite{Mak} for the unit disc case and \cite{Ll:98} for Lipschitz domains. Here, however, we use a superdyadic martingale, which is, essentially, a thinned dyadic martingale. It means that instead of the usual dyadic filtration $\mathcal{F}_k$ we use some subsequence of dyadic sigma-algebras $\mathcal{F}_{\alpha_k}$ where $\alpha_k$ depends on the weight $w$ (and is lacunary for slow growing $w$). The main reason for the transition from the dyadic to the superdyadic martingale approximation here is that the quadratic function of the superdyadic martingale is much easier to estimate (similar ideas were used in \cite{LM:12}).
\section{Proof of Theorems \ref{t:th1} and \ref{t:th2}}\label{s:prf}
\subsection{Main approximation lemma}\label{ss:ml}
Fix a Lipschitz function $\phi:\mathbb{R}^n\mapsto\mathbb{R}$, a doubling weight $w$ and a function $u$ in $h^{\infty}_w(\Omega_\phi)$. 
Given two functions $f$ and $g$, we say that $f\lesssim g$ if there is a positive constant $C = C(w,n,\|\phi'\|_{\infty}, \|u\|_{w,\infty})$ such that $f\leq Cg$. We write $f\sim g$ if $f\lesssim g$ and $g\lesssim f$ simultaneously. Consider a positive decreasing sequence $\{s_k\}_{k=0}^{\infty}$ such that 
\begin{equation}\notag
w(s_k) = 2^k,\quad k\in \mathbb{Z}_+,
\end{equation}
and put
\begin{equation}\notag
\alpha_0 = 0,\;\alpha_k = -\left[\frac{\log s_k}{\log2}\right],\quad k\in\mathbb{N}.
\end{equation}
It follows from the doubling property \eqref{e:wdc} that $w(2^{-\alpha_k})\sim 2^{k}$. Consider $I(x,\delta)$ defined in \eqref{e:idd}.
The approximation of $I(x,\delta)$ by martingales is provided by the following lemma
\begin{lemma}\label{l:ml}
Assume that $u \in h^{\infty}_w(\Omega_{\phi})$. Then for every $x_0\in\mathbb{R}^n$ there exists a probability measure $\mu$ on $Q(x_0)$ and a (super)dyadic martingale $\Lambda = \{\Lambda_k, \mathcal{F}_{\alpha_k},\mu\}_{k=0}^{\infty}$ on $Q(x_0)$ such that $\mu$ is absolutely continuous with respect to the Lebesgue measure on $Q(x_0)$ and for every $k\in\mathbb{Z}_+$
\begin{subequations}
\begin{eqnarray}
\label{e:li.1}&|\Lambda_k(x) - I(x,s_k)| \lesssim 1,\\
\label{e:li.2}&|\Lambda_k(x) - \Lambda_{k+1}(x)| \lesssim 1,\quad x\in Q(x_0).
\end{eqnarray}   
\end{subequations}
\end{lemma}
\subsection{How to deduce Theorem \ref{t:th1}}\label{ss:mtprf}
Assuming that Lemma \ref{l:ml} holds, we proceed by the standard argument. Fix any $x_0\in \mathbb{R}^n$ and put 
\[E = \{x\in Q(x_0): \lim_{m\rightarrow\infty}|\langle\Lambda\rangle_m|(x) < \infty\}.
\]
The inequality \eqref{e:li.2} implies that $\langle\Lambda\rangle_m^2 \lesssim m,\; m\geq 1$. Applying the law of the iterated logarithm for martingales to $\Lambda$ (see, for example, Theorem \textbf{3.0.2} in \cite{BM:98}), we get
\begin{equation}\notag
\limsup_{m\rightarrow\infty}\frac{|\Lambda_m|(x)}{\sqrt{m\log\log m}} \lesssim 1\quad \;\mu\; a.e.\; x\in Q(x_0)\setminus E.
\end{equation}
It is well known that for $\mu$ almost every $x\in E$ the sequence $\{\Lambda_m(x)\}$ is bounded, so \eqref{e:li.1} implies that the sequence $\{I(x,s_m)\}$ is bounded $\mu\; a.e.$ on $E$ as well. It follows that
\begin{equation}\label{e:ilil}
\limsup_{m\rightarrow\infty}\frac{|I(x,s_m)|}{\sqrt{m\log\log m}} \lesssim 1\quad \mu\; a.e.\; x\in Q(x_0).
\end{equation}
Now for $s_{m}\leq \delta\leq s_{m-1}$ and $x\in Q(x_0)$ we have
\begin{multline*}
\left|I(x,s_m) - I(x,\delta)\right| \leq \int_{s_m}^{\delta}|u(x,\phi(x)+y)|\,d\left(\frac{1}{w(y)}\right)\\ \lesssim \log w(s_m) - \log w(s_{m-1}) = 1,
\end{multline*}
also, clearly, $w(\delta) \geq w(s_{m-1}) = \frac{1}{2}w(s_m)$. Combined with \eqref{e:ilil} and the fact that $\mu$ is absolutely continuous with respect to the Lebesgue measure, it gives us
\begin{equation}\notag
\limsup_{\delta\rightarrow\infty}\frac{|I(x,\delta)|}{\sqrt{\log w(\delta)\log\log\log w(\delta)}} \lesssim 1, \quad a.e.\; x\in Q(x_0).
\end{equation}
The inequality \eqref{e:alil} follows immediately.
\subsection{Proof of the lemma \ref{l:ml}: auxiliary function $H$}\label{ss:hir}
The approximation of $I(x,\theta)$ by a Bloch function is covered by the following lemma
\begin{lemma}\label{l:bl}
Assume that $u \in h^{\infty}_w(\Omega_{\phi})$. Put
\begin{equation}\label{e:had}
H(x,t) = \int_{0}^1u(x,t+y)\,d\left(\frac{1}{w(y)}\right),\quad (x,t)\in \Omega_{\phi}.
\end{equation}
Then $H$ belongs to $\mathcal{B}(\Omega_{\phi})$ and $\|H\|_{\mathcal{B}}\lesssim 1$. Moreover
\begin{equation}\label{e:hia}
|H(x,\phi(x)+\theta) - I(x,\theta)| \lesssim 1,\quad x\in\mathbb{R}^n, 0<\theta\leq 1.
\end{equation}
\end{lemma}
\begin{proof}
First we note that $H$ is harmonic in $\Omega_{\phi}$ (it is the average of harmonic functions). We proceed by proving the following inequality
\begin{equation}\label{e:gru}
|\nabla u|(x,\phi(x)+\theta) \lesssim \frac{w(\theta)}{\theta},\quad x\in\mathbb{R}^n,\; \theta >0.
\end{equation}
Fix any positive $\theta$. Since $\phi$ is a Lipschitz function, we see that $\dist((x,\theta+\phi(x)),\partial\Omega_{\phi})\sim \theta$ for any $x\in\mathbb{R}^n$ and positive $\theta$. It follows from \eqref{e:wdc} that for $y\geq \frac{\theta}{2}$ we have
\begin{equation}
|u(x,\phi(x)+y)| \leq w(\dist(x,y+\phi(x)),\partial\Omega_{\phi})) \lesssim w\left(\frac{\theta}{2}\right),
\end{equation}
so there exists a constant $C = C(n,\phi,w)$ such that
\begin{equation*}
0\leq u(x,y) + Cw\left(\frac{\theta}{2}\right) \leq (C+1)w\left(\frac{\theta}{2}\right),\quad (x,y)\in \Omega_{\phi+\frac{\theta}{2}}.
\end{equation*}
Then the estimate \eqref{e:gru} follows from \eqref{e:0}, \eqref{e:wdc}, and the Harnack inequality.
For $(x,\theta)\in \Omega_{\phi}$ \eqref{e:gru} implies
\begin{multline*}
\left|\nabla H\right|(x,\phi(x) + \theta) \leq \int_{0}^1|\nabla u|(x,\phi(x) + \theta+y)\,d\left(\frac{1}{w(y)}\right)\\ \lesssim \int_{0}^1\frac{w(\theta+y)}{\theta+y}\,d\left(\frac{1}{w(y)}\right)   =
\int_{0}^{\theta}\frac{w(\theta+y)}{\theta+y}\,d\left(\frac{1}{w(y)}\right) + \int_{\theta}^1\frac{w(\theta+y)}{\theta+y}\,d\left(\frac{1}{w(y)}\right).
\end{multline*}
Since the function $\frac{w(y)}{y}$ is decreasing, we have
\begin{equation}\notag
\int_{0}^{\theta}\frac{w(\theta+y)}{\theta+y}\,d\left(\frac{1}{w(y)}\right) \leq \int_{0}^{\theta}\frac{w(\theta)}{\theta}\,d\left(\frac{1}{w(y)}\right) = \frac{1}{\theta}.
\end{equation}
On the other hand, 
\begin{multline}\label{e:intm}
\int_{\theta}^1\frac{w(\theta+y)}{\theta+y}\,d\left(\frac{1}{w(y)}\right) \leq \int_{\theta}^1\frac{w(y)}{y}\,d\left(\frac{1}{w(y)}\right)\\ \leq 
\sum_{k=0}^{[\log\frac{1}{\theta}]}\int_{2^k\theta}^{2^{k+1}\theta}\frac{1}{y}d\log\frac{1}{w(y)} \leq \sum_{k=0}^{[\log\frac{1}{\theta}]}\frac{1}{2^{k}\theta}\int_{2^k\theta}^{2^{k+1}\theta}d\log\frac{1}{w(y)} \\ \leq
\frac{1}{\theta}\sum_{k=0}^{[\log\frac{1}{\theta}]}2^{-k}(\log w(2^k\theta) - \log w(2^{k+1}\theta))\\ \leq \frac{1}{\theta}\sum_{k=0}^{[\log\frac{1}{\theta}]}2^{-k}(\log \left(Dw(2^{k+1}\theta)\right) - \log w(2^{k+1}\theta)) \\ \leq \frac{\log D}{\theta}\sum_{k=0}^{[\log\frac{1}{\theta}]}2^{-k} \lesssim \frac{1}{\theta}.
\end{multline}
Gathering the estimates, we arrive at
\begin{equation}\notag
\left|\nabla H\right|(x,\phi(x)+\theta) \lesssim \frac{1}{\theta} \sim \frac{1}{\dist((x,\phi(x)+\theta), \partial\Omega_{\varphi})},
\end{equation}
and we get the first part of the lemma.\par
To prove \eqref{e:hia} we write
\begin{multline*}
H(x,\phi(x)+\theta) - I(x,\theta)\\ = \int_{0}^1u(x,\phi(x) + \theta+y)\,d\left(\frac{1}{w(y)}\right) - \int_{\theta}^1u(x,y)\,d\left(\frac{1}{w(y)}\right)\\ = 
\int_{0}^{\theta}u(x,\phi(x) + \theta+y)\,d\left(\frac{1}{w(y)}\right)\\ + \int_{\theta}^1\left(u(x,\phi(x) + \theta+y) - u(x,\phi(x+y))\right)\,d\left(\frac{1}{w(y)}\right).
\end{multline*}
Following the same reasoning as above, we see that
\begin{equation}\notag
\left|\int_{0}^{\theta}u(x,\phi(x) + \theta+y)\,d\left(\frac{1}{w(y)}\right)\right| \leq w(\theta)\int_{0}^{\theta}\,d\left(\frac{1}{w(y)}\right) = 1.
\end{equation}
Again, \eqref{e:gru} implies that
\begin{multline*}
\left|\int_{\theta}^1\left(u(x,\phi(x) + \theta+y) - u(x,\phi(x+y))\right)\,d\left(\frac{1}{w(y)}\right)\right|\\ \leq \int_{\theta}^1\int_{y}^{y+\theta}|\nabla u|(x,\phi(x)+s)\,ds\,d\left(\frac{1}{w(y)}\right)\\\leq
\int_{\theta}^1\int_{y}^{y+\theta}\frac{w(s)}{s}\,ds\,d\left(\frac{1}{w(y)}\right)\leq \int_{\theta}^1w(y)\frac{\theta}{y}\,d\left(\frac{1}{w(y)}\right) \lesssim 1,
\end{multline*}
just like in \eqref{e:intm}. Combining these two inequalities we get \eqref{e:hia}.
\end{proof}
\subsection{Proof of Lemma \ref{l:ml}: dyadic martingale}\label{ss:dma}
Once we obtained the intermediate approximation of $I$ by a Bloch function, we can proceed to martingales. It is well known (see, for example, \cite{Mak}) that the Bloch functions in the unit disc can (up to a constant error) be viewed as dyadic martingales. The case of Lipschitz domains was considered by Llorente, Corollary \textbf{2} in \cite{Ll:98} is the main instrument in the following argument.

Fix any point $x_0\in\mathbb{R}^n$ and let
\begin{equation}\notag
\begin{split}
&A = \|\phi'\|_{\infty}\sqrt{n},\;\lambda = 8+ \frac{1}{A},\\
&\Omega_1 = \{(x,y): x\in \lambda Q(x_0):\phi(x)\leq y \leq \phi(x) + \lambda A\}.
\end{split}
\end{equation}
The following proposition holds
\begin{proposition}[Corollary \textbf{2}, \cite{Ll:98}]\label{p:c2ll}
If $v\in \mathcal{B}(\Omega_1)$ then there exists a dyadic martingale $\mathcal{M} = \{M_k,\mathcal{F}_k(x_0),\omega\}$ in $Q(x_0)$ and a positive constant $C = C(\phi,n)$ such that $\omega$ is absolutely continuous with respect to the Lebesgue measure on $Q(x_0)$, and if $A2^{-(k+1)}\leq t\leq A2^{-k}$ then for every $k\in\mathbb{N}$ and $x\in Q(x_0)$
\begin{subequations}
\begin{eqnarray}
\label{e:mv.1}& |M_k(x) - v(x,\phi(x)+t)| \leq C\|v\|_{\mathcal{B}},\\
\label{e:mv.2}&|M_{k+1}(x) - M_k(x)| \leq C\|v\|_{\mathcal{B}}.
\end{eqnarray}
\end{subequations}
\end{proposition}
We apply this proposition to $H$ and put
\begin{equation}\notag
\Lambda = \{\Lambda_k,\mathcal{F}_{\alpha_k}(x_0),\omega\} := \{M_{\alpha_k},\mathcal{F}_{\alpha_k}(x_0),\omega\}.
\end{equation}
Now we prove \eqref{e:li.1} and \eqref{e:li.2}. It follows from \eqref{e:mv.1} and \eqref{e:hia} that
\begin{multline*}
|\Lambda_k(x) - I(x,s_k)| = |M_{\alpha_k}(x) - I(x,s_k)| \\ \leq
|M_{\alpha_k}(x) - H(x,\phi(x) + A2^{-\alpha_k})| + |H(x,\phi(x) + A2^{-\alpha_k}) - H(x,\phi(x) + s_k)| \\+ |H(x,\phi(x) + s_k) - I(x,s_k)| \lesssim
1 + \int_{s_k} ^{A2^{-\alpha_k}}\left|\nabla H(x,\phi(x)+y)\right|\,dy \\ \lesssim 1,\quad x\in Q(x_0),
\end{multline*}
since $\|H\|_{\mathcal{B}} \lesssim 1$, and we get \eqref{e:li.1}. To obtain \eqref{e:li.2} we note that
\begin{multline*}
|\Lambda_k(x) - \Lambda_{k+1}(x)| = |M_{\alpha_k}(x) - M_{\alpha_{k+1}}(x)| \leq
|M_{\alpha_k}(x) - I(x,s_k)|\\ + |I(x,s_k) - I(x,s_{k+1})|
 + |I(x,s_{k+1}) - M_{\alpha_{k+1}}(x)|\\ \lesssim 1 + |I(x,s_k) - I(x,s_{k+1})|.
\end{multline*}
Clearly,
\begin{equation}\notag
|I(x,s_k) - I(x,s_{k+1})| \leq w(s_k)\int_{s_{k+1}}^{s_k}\,d\left(\frac{1}{w(y)}\right) = 2^k(2^{-k} - 2^{-k-1}) = 2,
\end{equation}
and the inequality \eqref{e:li.2} follows.
\subsection{Proof of Theorem \ref{t:th2}}\label{ss:crlr} The proof is standard. We apply the usual ice-cream cone construction to $\Sigma$, i.e. consider the domain
\begin{equation}\notag
\Omega = \bigcup_{x\in \Sigma} \Gamma(x,M),
\end{equation}
where $\Gamma(x,M)$ is the cone with vertex $x$ and aperture $M$, $\Gamma(x,M) = \{(\tilde{x},y)\in\mathbb{R}^{n+1}_+: |\tilde{x}-x|\leq My \}$. Clearly $\Omega$ is the area above the graph of some Lipschitz function $\phi$ with $\|\phi'\|_{\infty} = \frac{1}{M}$, so that $\Omega = \Omega_{\phi}$. The condition \eqref{e:cgr} then implies that
\begin{equation}\notag
|u(x,y)| \leq Kw(y) \lesssim w(\dist((x,y),\partial \Omega)),\quad (x,y)\in \Omega,
\end{equation}
and we can apply Theorem \ref{t:th1} to obtain
\begin{equation}\notag
\limsup_{\delta\rightarrow0}\frac{I(x,\delta)}{\sqrt{\log w(\delta)\log\log\log w(\delta)}} \leq C,\quad a.e.\; x\in\mathbb{R}^n.
\end{equation}
Theorem \ref{t:th2} follows immediately.
\section{An example}\label{s:ex}
\subsection{}
In the proof of Theorem \ref{t:th1} we introduced the harmonic function $H$ which is shown to be a Bloch function. In addition, the estimate \eqref{e:hia} implies that $H \in h^{\infty}_{\log w}(\Omega_{\phi})$, and that the LIL  in \eqref{e:alil} holds for $H$ as well,
\begin{equation}\label{e:blil}
\limsup_{\delta\rightarrow0}\frac{H(x,\delta)}{\sqrt{\log w(\delta)\log\log\log w(\delta)}} \lesssim 1, \quad a.e.\; x\in\mathbb{R}^n.
\end{equation}
To obtain this estimate we used the special nature of $H$, namely that it was constructed on $u\in h^{\infty}_w(\Omega_{\phi})$. It is then natural to ask if an \textit{arbitrary} function $v \in h^{\infty}_{w_0}(\Omega_{\phi})\bigcap \mathcal{B}(\Omega_{\phi})$ satisfies the following LIL
\begin{equation}\label{e:clil}
\limsup_{\delta\rightarrow0}\frac{H(x,\delta)}{\sqrt{w_0(\delta)\log\log w_0(\delta)}} \lesssim 1,\quad a.e. \;x\in\mathbb{R}^n.
\end{equation}
The answer to this question is negative as provided by the following proposition.
\begin{proposition}\label{p:ex}
Let $w_0(y) = \log\log\frac{e}{y}+1,\; y\in (0,1]$. There exists a constant $A>0$, a function $v\in h^{\infty}_{w_0}(\mathbb{R}^2_+)\bigcap \mathcal{B}(\mathbb{R}^2_+) $, a number $k_0\in\mathbb{N}$ and a sequence $\{y_k\}_{k=k_0}^{\infty}\rightarrow 0$ such that
\begin{equation}\label{e:ex}
\lambda_1\left(\{t\in[0,1]:\, |v(t,y_k)| \geq \frac{w_0(y_k)}{A}\}\right) \geq \frac{1}{10}.
\end{equation}
\end{proposition}
It is known that a function in $h^{\infty}_w(\mathbb{R}^2_+)$ can grow as fast as $w$ only on small part of vertical rays $\{x+iy\},\; y\in \mathbb{R}_+$, however it can attain the maximal growth on the subsets of those rays for a.e. $x\in \mathbb{R}$ (see \cite{LM:12}, \cite{BLMT:10}). Unfortunately, we can not use the example provided there, since it is constructed as a lacunary trigonometric series, for which, as it can be shown, \eqref{e:blil} holds if it belongs to the Bloch class. \\
\begin{proof}
Given two real-valued functions $f,g\in L^2(\mathbb{R}^n)$ we denote the scalar product $\int_{\mathbb{R}^n}f(t)g(t)dt$ by $\langle f,g\rangle$. Consider a function $\varphi:\mathbb{R}\mapsto\mathbb{R}$ such that $\supp \varphi\subset [0,1],\;\varphi\in C^{10},\; \|\varphi\|_{\infty}\leq 1$. We also require that $\int_{\mathbb{R}}\varphi(t)\,dt =0$ and $\langle\varphi,\psi\rangle \neq 0$ (where $\psi$ is a Haar wavelet mentioned earlier). For example we can take the suitable renormalization of the Daubechies wavelet (or any other smooth wavelet with compact support that satisfies our conditions). By $P_y$ we denote the Poisson kernel for the halfplane, $P_y(t) = \frac{y}{\pi(y^2+t^2)},\; y>0,\,t\in\mathbb{R}$. \par
The idea is to obtain a functional series of the form
\begin{equation}\label{e:wavjs}
\sum_{j=0}^{k}\sum_{I\in \Delta_j}c_{I}\varphi_{I}(t) := \Phi_k(t),\quad t\in\mathbb{R},
\end{equation}
that satisfies properties similar to those in the statement, and then prove that the corresponding Bloch function provides the required example. To elaborate we first construct $\Phi_k$ and an increasing sequence $\{\beta_j\}_{j=1}^{\infty}\subset \mathbb{Z}_+$ in such a way that  we have
\begin{subequations}\label{e:mblgr}
\begin{eqnarray}
\label{e:mblgr.1}&\|\Phi_k - \Phi_{k-1}\|_{\infty} \leq 1,\\
\label{e:mblgr.2}&\|\Phi_{k}\|_{\infty} \leq w_0(2^{-k})+2,\\
\label{e:mblgr.3}&\lambda_1\left(\{t\in(0,1]: |\Phi_{\beta_{k}}|(t) \geq \frac{w_0(2^{-\beta_{k}})}{4}\}\right) \geq \frac{1}{10},
\end{eqnarray}
\end{subequations}
for any integer $k\geq k_0$.\par
The property \eqref{e:mblgr.1} is an analogue of the Bloch condition, \eqref{e:mblgr.2} is the growth restriction, and \eqref{e:mblgr.3} corresponds to \eqref{e:ex} (so that there is no LIL for $\Phi_k$ with $w_0$).
\subsection{Construction of $\{\beta_j\}$ and $\Phi_k$}\label{ss:sbeta}
First we chose $a\in\mathbb{N}$ such that
 $2^{-a+1}\|\varphi'\|_{\infty} \leq \frac14\left|\langle\varphi,\psi\rangle\right|$.
Now chose a natural $j_0 \geq 4\|\varphi'\|_{\infty}+4$ and an increasing sequence $\beta_j\in\mathbb{N}$ in such a way that
\begin{equation}\label{e:jn}
\begin{split}
&\beta_1 = 0,\;\frac{\beta_j}{a}\in\mathbb{N},\\
&j-1\leq w_0(2^{-\beta_j})\leq j,\\
&\left(\frac{\beta_j - \beta_{j-1}}{a}-1\right)\langle\varphi,\psi\rangle^2 \geq 4j^2, \quad j\geq j_0.
\end{split}
\end{equation}
It is not hard to verify that such choice is possible (we remind that $w_0(y) = \log\log\frac{e}{y}+1$).\par
The functions $\Phi_k$ are constructed via double induction, first on $j$, and then on $m$ between $\beta_j$ and $\beta_{j+1}$. Put $\Phi_0(t) = \varphi(t)$. Assume now that we obtained $\Phi_{\beta_j}$ for some $j\in\mathbb{N}$. Consider all the intervals $I\in \Delta_{\beta_j}$ such that $\sup_{t\in I}|\Phi_{\beta_j}|(t) > j$, we denote the set of these intervals by $\mathcal{E}^{\beta_j}_j$. Now suppose that we have constructed $\Phi_{m-1}$ and $\mathcal{E}_{j}^{m-1}$ for some $m,\, \beta_{j}+1\leq m \leq \beta_{j+1}$. If $\frac{m}{a}\in\mathbb{Z}$, then for $I\in \Delta_m$ and $t\in I$ let
\begin{subequations}\label{e:lmdr}
\begin{eqnarray}
\label{e:lmdr.3}&\Phi_{m}(t) = \Phi_{m-1}(t),\quad t\in \bigcup_{J\in \mathcal{E}^{m-1}_j}J,\\
\label{e:lmdr.2}&\Phi_{m}(t) = \Phi_{m-1}(t) + \varphi_{I}(t),\quad t\notin \bigcup_{J\in \mathcal{E}^{m-1}_j}J,\\
\label{e:lmdr.1}&\mathcal{E}^m_j = \mathcal{E}^{m-1}_j\bigcup\{J\in\Delta_m: \sup_{t\in J}|\Phi_m(t)| >j\}.
\end{eqnarray}
\end{subequations}
Otherwise we put
\begin{subequations}\label{e:lmdr1}
\begin{eqnarray}
\label{e:lmdr1.3}&\Phi_{m}(t) = \Phi_{m-1}(t),\quad t\in (0,1],\\
\label{e:lmdr1.1}&\mathcal{E}^m_j = \mathcal{E}^{m-1}_j.
\end{eqnarray}
\end{subequations}
Finally, put
\begin{equation}\notag
\mathcal{E}_j = \bigcup_{m=\beta_j}^{\beta_{j+1}-1}\mathcal{E}^{m}_j = \mathcal{E}^{\beta_{j+1}-1}_j.
\end{equation}
What we do here is, essentially, a stopping time procedure applied (instead of martingales as usual) to the functional series of the form like in \eqref{e:wavjs}. We see that if $I \in \mathcal{E}_j$, then the construction is stopped at this interval, and $\Phi_{\beta_{j+1}}(t) = \Phi_{m}(t),\; t\in I,\; m = \rank I$. If, on the other hand, $t \in (0,1]\setminus \bigcup_{J\in\mathcal{E}_j}J$, then the construction happens on every step (divisible by $a$) up until $\beta_{j+1}$.  The set $(0,1]\setminus\bigcup_{J\in\mathcal{E}_j}J$ can be decomposed into a disjoint union of intervals from $\Delta_{\beta_{j+1}}$, we denote the set of these intervals by $\mathcal{G}_j$.\\
Clearly $\Phi_m$ is of the form like in \eqref{e:wavjs}, moreover,
\begin{equation}\notag
\Phi_{\beta_{j+1}}(t) = \Phi_{\beta_j}(t) + \sum_{m=\beta_{j}+1}^{\beta_{j+1}}\sum_{J\in\Delta_m}c_J\varphi_J(t),\quad t\in (0,1],
\end{equation}
where $c_J =1$ only if $\frac{\rank J}{a} \in\mathbb{Z}$ \textit{and} there is no interval $I\in\mathcal{E}_j$ such that $I\supset J$, $c_J = 0$ otherwise. We also see that
\begin{equation}\notag
\begin{split}
&\sup_{t\in I}|\Phi_{\beta_{j+1}}|(t) > j,\quad I\in \mathcal{E}_j;\\
&|\Phi_{\beta_{j+1}}| \leq j+1, \quad j\in\mathbb{N}.
\end{split}
\end{equation}\par
We are left to check \eqref{e:mblgr.1}-\eqref{e:mblgr.3}. The condition \eqref{e:mblgr.1} follows straight from \eqref{e:lmdr}, since $\|\varphi_{I}\|_{\infty} = 1$ for any interval $I$. For any $k\in\mathbb{N}$ there exists $j_k\in \mathbb{N}$ such that $\beta_{j_k}\leq k \leq \beta_{j_k+1}-1$. We therefore have
\begin{equation}\notag
\|\Phi_{k}\|_{\infty} \leq j_k+1 \leq w_0(2^{-\beta_{j_k}}) +2 \leq w_0(2^{-k}) + 2,
\end{equation}
and we obtain \eqref{e:mblgr.2}. 

\subsection{Proof of \eqref{e:mblgr.3}: martingale decomposition}\label{ss:mblgr.3}
Pick any $j\geq j_0$ (we remind that $j_0$ was defined in \eqref{e:jn}). Since $j_0 \geq 4\|\varphi'\|_{\infty}+4$, we see that $\frac{j}{2} - \|\varphi'\|_{\infty} \geq \frac{j}{2} - \frac{j_0}{4} \geq \frac{j}{4}$, and, due to \eqref{e:jn} we have $\frac{j}{2} - \|\varphi'\|_{\infty} \geq \frac{w_0(2^{-\beta_j})}{4}$. It follows that to obtain \eqref{e:mblgr.3} it is enough to prove
\begin{equation}\label{e:lgi}
\lambda_1\left(\{t\in (0,1]:|\Phi_{\beta_{j+1}}(t)| \geq \frac{j}{2}-\|\varphi'\|_{\infty}\}\right) \geq \frac{1}{10}.
\end{equation}

The first step is to prove that $|\Phi_{\beta_{j+1}}|$ is "sufficiently large" on the intervals from $\mathcal{E}_j$, namely that for any $I\in \mathcal{E}_j$ we have
\begin{equation}\label{e:phijj}
|\Phi_{\beta_{j+1}}|(t) \geq j - 2\|\varphi'\|_{\infty},\quad t\in I.
\end{equation}
Indeed, for $m = \rank I$ we have
\begin{multline*}
|\Phi_m'|(t) \leq \sum_{J\in\Delta:\, t\in JI,\, \rank J \leq m}c_{J}\|\varphi_{J}'\|_{\infty}\\ = \|\varphi'\|_{\infty}\sum_{J\in\Delta:\, t\in J,\, \rank J \leq m}\frac{c_J}{\lambda_1(J)}\leq 2^{m+1}\|\varphi'\|_{\infty},\quad t\in (0,1],
\end{multline*}
since $|c_J|\leq 1$ for any $J\in \Delta$. Again we see that $\Phi_{\beta_{j+1}}(t) = \Phi_{m}(t)$ on $I$, therefore $\left|\sup_{t\in I}\Phi_m(t) - \inf_{t\in I}\Phi_m(t)\right| \leq \int_{I}|\Phi'_m(t)|\,dt\leq 2\|\varphi'\|_{\infty}$, and we get \eqref{e:phijj}.\par
Now we show that
\begin{equation}\label{e:gj}
\lambda_1\left(\bigcup_{J\in\mathcal{G}_j}J\right) \leq \frac{3}{4},
\end{equation}
combined with \eqref{e:phijj} it implies \eqref{e:lgi}. In order to do this consider the Haar decomposition of $\Phi_{\beta_{j+1}}$,
\begin{equation}\label{e:phihd}
\Phi_{\beta_{j+1}} = \sum_{m=0}^{\infty}\sum_{I\in \Delta_m}b_{I}\psi_{I},
\end{equation}
where $b_{I} = 2^{\rank I}\langle\Phi_{\beta_{j+1}},\psi_{I}\rangle =$ $2^{\rank I}\sum_{k=0}^{\beta_{j+1}}\sum_{J\in \Delta_k}c_J\langle\varphi_J,\psi_I\rangle$, and $c_{J}$ is either $0$ or $1$. Here we sum from $m=0$, since $\supp \Phi_{k}\subset [0,1]$ and $\int_{0}^1\Phi_k(t)\,dt=0$ for any $k\in\mathbb{Z}_+$. If we put
\begin{equation}\notag
\tilde{\Lambda}_k = \sum_{m=0}^{k}\sum_{I\in \Delta_m}b_{I}\psi_{I},
\end{equation}
we see that $\{\tilde{\Lambda}_k,\mathcal{F}_k,\lambda_1\}$ is a dyadic martingale on $(0,1]$. Since $\Phi_{\beta_{j+1}} \in C^{10}(\mathbb{R})$, the sum on the right-hand side in \eqref{e:phihd} converges to $\Phi_{\beta_{j+1}}$ uniformly on $\mathbb{R}$ as $k\rightarrow\infty$. It follows immediately that $\langle\tilde{\Lambda}\rangle_k$ converges uniformly to a bounded limit which we denote by $\langle\tilde{\Lambda}\rangle_{\infty}$.\\
Our goal here is to prove that the quadratic function of $\tilde{\Lambda}$ is "big" on the intervals from $\mathcal{G}_j$, so that we can use the standard dyadic martingale methods to estimate the size of $\bigcup_{J\in\mathcal{G}_j}J$. For any $k\in\mathbb{Z}_+$, the following equality holds
\begin{equation}\label{e:gli}
\int_{0}^1\langle\tilde{\Lambda}\rangle_k^2(t)\,dt = \int_{0}^1\tilde{\Lambda}_k^2(t)\,dt.
\end{equation}
Assume for a moment that we know that
\begin{equation}\label{e:qfl}
\langle\tilde{\Lambda}\rangle_{\infty}^2(t) \geq 4j^2,\quad t\in \bigcup_{J\in\mathcal{G}_j}J.
\end{equation}
Then \eqref{e:gli} implies that
\begin{equation}\notag
\begin{split}
&4j^2\cdot\lambda_1\left(\bigcup_{J\in\mathcal{G}_j}J\right) \leq \int_{\bigcup_{J\in\mathcal{G}_j}J}\langle\tilde{\Lambda}\rangle_{\infty}^2(t)\,dt \leq \int_0^1\langle\tilde{\Lambda}\rangle_{\infty}^2(t)\,dt \\=
&\int_0^1\tilde{\Lambda}_{\infty}^2(t)\,dt = \int_0^1\Phi_{\beta_{j+1}}^2(t)\,dt \leq (j+1)^2,
\end{split}
\end{equation}
and \eqref{e:gj} follows immediately. It remains to prove the estimate \eqref{e:qfl}.
\subsection{Proof of \eqref{e:mblgr.3}: inequality \eqref{e:qfl}}\label{ss:mblgr.4}
First we show that if $c_{I} = 1$ for some $I\in\Delta_m,\, \beta_{j} \leq m\leq \beta_{j+1}-1$, then 
\begin{equation}\label{e:hcl}
|b_{I}| \geq \frac12\left|\langle\varphi,\psi\rangle\right|.
\end{equation}
Fix such an interval $I$. For any $J\in \Delta_k,\; k\leq m$, the standard calculation gives
\begin{multline}\notag
|\langle\varphi_J,\psi_I\rangle|\\ = \left|\int_{\mathbb{R}}\varphi(2^kt-x_J)\psi(2^mt-x_I)\,dt\right| = \left|\int_{\mathbb{R}}\varphi(2^k(t-2^{-m}x_I)-x_J)\psi(2^mt)\,dt\right|\\ =
2^{-k}\left|\int_{\mathbb{R}}\varphi(t-2^{k-m}x_I - x_J)\psi(2^{m-k}t)\,dt\right|\\ = 2^{-k}\left|\int_{\mathbb{R}}\left(\varphi(t-2^{k-m}x_I - x_J) - \varphi(-2^{k-m}x_I - x_J)\right)\psi(2^{m-k}t)\,dt\right| \\=
2^{-k}\left|\int_{\mathbb{R}}\int_{-2^{k-m}x_I - x_J}^{t-2^{k-m}x_I - x_J}\varphi'(s)\,ds\psi(2^{m-k}t)\,dt\right| \leq 2^{-k}\|\varphi'\|_{\infty}\int_{\mathbb{R}}\left|t\psi(2^{m-k}t)\right|\,dt \\=
2^{k-2m}\|\varphi'\|_{\infty}\int_{\mathbb{R}}|t\psi(t)|\,dt \leq 2^{k-2m}\|\varphi'\|_{\infty}.
\end{multline}
Now we see that if $k> m$, then $\langle\varphi_{J},\psi_{I}\rangle =0$ for any $J\in \Delta_k$, and if  $k\leq m$, then there exists at most one $J\in \Delta_k$ such that $\langle\varphi_{J},\psi_{I}\rangle \neq 0$. We therefore have
\begin{multline}\notag
|b_{I}| = 2^{m}\left|\sum_{k=0}^{\beta_{j+1}}\sum_{J\in \Delta_k}c_J\langle\varphi_J,\psi_I\rangle\right| = 2^{m}\left|\sum_{k\leq m,\,J\in \Delta_k,\, J\supset I}c_J\langle\varphi_J,\psi_I\rangle\right|  \\ \geq
 2^{m}\left|\langle\varphi_{I},\psi_{I}\rangle\right| - 2^{m}\sum_{k\leq m-1,\,J\in \Delta_k,\, J\supset I}|c_J||\langle\varphi_J,\psi_I\rangle|\\ \geq
\langle\varphi,\psi\rangle - 2^m\sum_{k\leq m-1,\,J\in \Delta_k,\, J\supset I}|c_J|\|\varphi'\|_{\infty}2^{k-2m}\\ \geq \langle\varphi,\psi\rangle - \|\varphi'\|_{\infty}\sum_{k\leq m-1,\,J\in \Delta_k,\, J\supset I}2^{k-m}|c_{J}|.
\end{multline}
It follows from \eqref{e:lmdr.2} that if $c_{J}=1$ then $c_{J} = 0$ for $J\in\Delta_k$, $m-a< k \leq m-1$ (the decomposition of $\Phi_{\beta_{j+1}}$ has very sparse coefficients). Combined with the choice of $a$, it gives
\begin{equation}\notag
|\langle\varphi,\psi\rangle| - \|\varphi'\|_{\infty}\sum_{k\leq m-1,\,J\in \Delta_k,\, J\supset I}2^{k-m}|c_{J}| \geq \langle\varphi,\psi\rangle - \|\varphi'\|_{\infty}2^{-a} \geq \frac{1}{2}|\langle\psi,\varphi\rangle|,
\end{equation}
and we have \eqref{e:hcl}.\par
Fix any $I\in \mathcal{G}_j$. Again we note that $c_J=1$ for any $J\in \Delta_m$ such that $\frac{m}{a}\in \mathbb{Z}$, $J\supset I$ and $\beta_j\leq m\leq\beta_{j+1}-1$. Therefore \eqref{e:hcl} implies that $|b_{J}| \geq \frac12\left|\langle\varphi,\psi\rangle\right|$ for such intervals $J$, and  due to \eqref{e:jn} we have
\begin{multline}\notag
\langle\tilde{\Lambda}\rangle_{\infty}^2(t) \geq \sum_{\beta_j\leq m\leq\beta_{j+1}-1,\, \frac{m}{a}\in\mathbb{Z},\,J\in \Delta_m,\, t\in J}|b_{J}|^2\\ \geq \frac14\left(\frac{\beta_{j+1}-\beta_j}{a}-1\right)\left|\langle\varphi,\psi\rangle\right|^2 \geq 100j^2,
\end{multline}
for $t\in I$, and we get \eqref{e:qfl}.
\subsection{How to create a Bloch function from $\Phi_j$}\label{ss:bloch}
Let
\begin{equation}\notag
v_k(x,y) = \left(\Phi_{k}*P_y\right)(x),\quad x\in\mathbb{R},\; k\geq k_0.
\end{equation}
First we show that $v_k\rightarrow v$ as $k\rightarrow\infty$, where $v$ is a harmonic function.\par
 Fix any $y>0$.  Since $c_I $ is either $0$ or $1$, we have for natural $m\leq n$
\begin{multline*}
\left|v_m(x,y) - v_n(x,y)\right| \\= \left|\sum_{j=m+1}^{n}\sum_{I\in \Delta_j}c_{I}\varphi_{I}*P_y\right|(x) \leq \sum_{j=m+1}^{n}\sum_{I\in \Delta_j}|c_{I}|\left|\int_{\mathbb{R}}\varphi_{I}(t)P_y(x-t)\,dt\right|\\ \leq
\sum_{j=m+1}^{n}\sum_{i=0}^{2^j-1}\left|\int_{\mathbb{R}}\varphi(2^jt)P_y\left(x-t-\left(i+\frac12\right)2^{-j}\right)\,dt\right|. 
\end{multline*}
A standard calculation gives
\begin{equation}\label{e:est}
\begin{split}
&\sum_{i=0}^{2^j-1}\left|\int_{\mathbb{R}}\varphi(2^jt)P_y\left(x-t-\left(i+\frac12\right)2^{-j}\right)\,dt\right| \\=
&\sum_{i=0}^{2^j-1}\left|\int_{\mathbb{R}}\varphi(2^jt)\left(P_y\left(x-t-\left(i+\frac12\right)2^{-j}\right)-P_y\left(x-\left(i+\frac32\right)2^{-j}\right)\right)\,dt\right|\\ \leq
&\sum_{i=0}^{2^j-1}\left|\int_{0}^{\frac{2^{-j}}{y}}\varphi(2^jyt)\left(P\left(\frac{x}{y}-t-\frac{i+\frac12}{y}2^{-j}\right)-P\left(\frac{x}{y}-\frac{i+\frac32}{y}2^{-j}\right)\right)\,dt\right|\\ \leq&
\sum_{i=0}^{2^j-1}\left|\int_0^{\frac{2^{-j}}{y}}\varphi(2^jyt)\int_{\frac{x}{y}-\frac{i+\frac32}{y}2^{-j}}^{\frac{x}{y}-t-\frac{i+\frac12}{y}2^{-j}}|P'|(s)\,ds\,dt\right|\\ &\leq \int_0^{\frac{2^{-j}}{y}}|\varphi(2^jyt)|\sum_{i=0}^{2^j-1}\int_{\frac{x}{y}-\frac{i+\frac32}{y}2^{-j}}^{\frac{x}{y}-t-\frac{i+\frac12}{y}2^{-j}}|P'|(s)\,ds\,dt\\ \leq
&\int_0^{\frac{2^{-j}}{y}}|\varphi(2^jyt)|\int_{\mathbb{R}}|P'|(s)\,ds\,dt \leq C\frac{2^{-j}}{y},\quad x\in\mathbb{R},\; y>0,\; j\in\mathbb{N}.
\end{split}
\end{equation}
We therefore have
\begin{equation}\label{e:est1}
\left|v_m(x,y) - v_n(x,y)\right|\leq C\sum_{j=m+1}^{n}\frac{2^{-j}}{y},\quad x\in\mathbb{R},\; y>0,
\end{equation}
and the uniform convergence follows immediately.\par
Next we show that $v$ satisfies the $h^{\infty}_{w_0}$ growth condition.
For $y\geq 2^{-k}$ \eqref{e:est1} implies that
\begin{equation}\notag
|v(x,y) - v_k(x,y)| \leq C,\quad x\in\mathbb{R}.
\end{equation}
Combined with \eqref{e:mblgr.2} and definition of $w_0$ this implies that
\begin{multline*}
|v(x,y)| \leq C+|v_k(x,y)| = C + \left|\Phi_{k}*P_y\right|(x)\\ \leq C + w_0(2^{-k}) \leq Cw_0(y),\quad x\in\mathbb{R},\; 2^{-k+1}\geq\;y\geq 2^{-k},\; k\in\mathbb{N},
\end{multline*}
and, therefore, $v\in h^{\infty}_{w_0}$.\par
Now we prove that $v\in\mathcal{B}(\mathbb{R}^2_+)$. Fix any positive $y\leq 1$ and $m\in\mathbb{Z}$ such that $2^{-m+1}\geq y \geq 2^{-m}$. We have
\begin{equation*}
\left|\nabla v(x,y)\right| \leq \left|\nabla v(x,y) - \nabla v_m(x,y)\right| + \left|\nabla v_m(x,y)\right|,\quad x\in\mathbb{R}.
\end{equation*}
Repeating the estimate in \eqref{e:est} verbatim we get for any $x\in \mathbb{R}$,
\begin{multline}\label{e:estp}
\left|\nabla v(x,y) - \nabla v_m(x,y)\right|\\ \leq \left|\frac{\partial}{\partial y}v(x,y) - \frac{\partial}{\partial y}v_m(x,y)\right| + \left|\frac{\partial}{\partial x}v(x,y) - \frac{\partial}{\partial x}v_m(x,y)\right| \\=
\left|\sum_{j=m+1}^{\infty}\sum_{I\in \Delta_j}c_{I}\varphi_{I}*\left(\frac{\partial}{\partial y}P_y\right)\right|(x) + \left|\sum_{j=m+1}^{\infty}\sum_{I\in \Delta_j}c_{I}\varphi_{I}*\left(\frac{\partial}{\partial x}P_y\right)\right|(x)\\  \leq \frac{C2^{-m}}{y^2} \leq \frac{C}{y}.
\end{multline}
Recall that $\varphi \in C^{10}(\mathbb{R})$ and that for any two different intervals $I, J \in \Delta_j$ the supports of $\varphi_I$ and $\varphi_J$ are disjoint. A simple rescaling gives
\begin{equation}\notag
\left|\sum_{I\in \Delta_j}c_{I}\varphi_{I}*\left(\frac{\partial}{\partial x}P_y\right)\right|(x) + \left|\sum_{I\in \Delta_j}c_{I}\varphi_{I}*\left(\frac{\partial}{\partial y}P_y\right)\right|(x) \leq C2^j,\quad x\in\mathbb{R},\, y>0,\, j\in\mathbb{N}.
\end{equation}
It follows that
\begin{multline*}
\left|\frac{\partial}{\partial x}v_m(x,y)\right| + \left|\frac{\partial}{\partial y}v_m(x,y)\right| = \left|\Phi_{m}*\left(\frac{\partial}{\partial x}P_y\right)\right|(x) + \left|\Phi_{m}*\left(\frac{\partial}{\partial y}P_y\right)\right|(x)\\ \leq  C\sum_{j=0}^m2^j = C2^{m+1} \leq \frac{C}{y},\quad x\in\mathbb{R}.
\end{multline*}
This estimate and \eqref{e:estp} imply that $v\in \mathcal{B}(\mathbb{R}^2_+)$.\par
It remains to prove \eqref{e:ex}. Fix any $k\geq k_0$. Since $\|\Phi'_{\beta_k}\|_{\infty}\leq 2^{\beta_k+1}\|\varphi'\|_{\infty}$, we see that for any $x\in (0,1]$ such that $\left|\Phi_{\beta_k}(x)\right| \geq \frac{k-1}{2}$, there exists an interval $I_x = [x-\rho_k,x+\rho_k]$, $\rho_k = \frac{2^{-\beta_k-2}}{\|\varphi'\|_{\infty}}$, such that $\left|\Phi_{\beta_k}(t)\right| \geq \frac{k-3}{2},\; t\in I_x$. Clearly then
\begin{equation}\notag
\int_{\mathbb{R}\setminus I_x}P_{y}(x-t)\,dt \leq \frac{1}{4},
\end{equation}
for $0<y\leq \frac{\rho_k}{10}$.
Now if we fix such an $x$ and put $y_k = \frac{2^{-\beta_k-2}}{10\|\varphi'\|_{\infty}}$, we have
\begin{multline*}
|v_{\beta_k}(x,y)| = \left|\Phi_{\beta_k}*P_y\right|(x)\\ \geq \left|\int_{I_x}\Phi_{\beta_k}(t)P_y(x-t)\,dt\right| - \left|\int_{\mathbb{R}\setminus I_x}\Phi_{\beta_k}(t)P_y(x-t)\,dt\right|\\ \geq
\frac{k-3}{2} - \|\Phi_{\beta_k}\|_{\infty}\int_{\mathbb{R}\setminus I_x}P_y(x-t)\,dt \geq
 \frac{k-3}{2} - \frac{k+2}{4}\\ = \frac{k-4}{2} \geq \frac{w_0(2^{-\beta_k})-4}{2} ,
\end{multline*}
so that
\begin{equation}\notag
\lambda_1\left(\left\{x\in (0,1]: \left|v_{\beta_k}\right|(x,y) \geq \frac{w_0(2^{-\beta_k})-4}{2}\right\}\right) \geq \frac{1}{10},\quad 0<y\leq y_k.
\end{equation}
Again, following \eqref{e:est}, we obtain
\begin{equation}
|v_{\beta_k}(x,y_k) - v(x,y_k)| \leq C_0\|\varphi'\|_{\infty}.
\end{equation}
The doubling property of $w_0$ implies 
\begin{equation}\notag
w_0(y_k) = w_0\left(\frac{2^{-\beta_k-2}}{10\|\varphi'\|_{\infty}}\right) \leq C_1\frac{w_0(2^{-\beta_k})-4}{2} - C_0\|\varphi'\|_{\infty} \leq C\frac{w_0(2^{-\beta_k})-4}{2}
\end{equation}
for $k$ large enough.
We therefore have
\begin{equation}\notag
\lambda_1\left(\left\{x\in (0,1]: \left|v\right|(x,y_k) \geq \frac{w_0(y_k)}{C}\right\}\right) \geq \frac{1}{10},\quad y_k = \frac{2^{-\beta_k-2}}{10\|\varphi'\|_{\infty}},
\end{equation}
which is \eqref{e:ex}.
\end{proof}
The way we did the construction of $v$ is, probably, not the most effective one. Unfortunately we could not use here the dyadic martingale methods, as described, for example, in \cite{GM:05}. Instead we decided to employ the wavelet-like series for Bloch functions (see \cite{M:01} for the description of $\mathcal{B}(\mathbb{R}^2_+)$ in terms of wavelet representation).\par
Note that the averaging process $u(x,\delta)\rightarrow H(x,\delta) = \int_{0}^1u(x,y+\delta)\,d\frac{1}{w(y)}$ can be viewed as an application of some multiplier $M$ to the boundary values of $u$,
\begin{equation}\notag
\widehat{Mf}(\tau) = \widehat{f}(\tau)\int_0^1e^{-2\pi y|\tau|}d\frac{1}{w(y)},\quad \tau\in\mathbb{R},
\end{equation}
where $f = u(\cdot,0)$ (these boundary values exist in some sense, at least as a distribution). The doubling condition \eqref{e:wdc} implies that 
\begin{equation}\notag\
\int_0^1e^{-2\pi y|\tau|}d\frac{1}{w(y)}\sim \frac{1}{w(\frac{1}{|\tau|})},\quad |\tau| > 0,
\end{equation}
so we basically divide the Fourier transform of $u(\cdot,0)$ by $w$. It would be interesting to find out the image of $M$, we see at least that $Mu\in h^{\infty}_w(\mathbb{R}^{n+1}_+)\bigcap\mathcal{B}(\mathbb{R}^{n+1}_+)$ if $u\in h^{\infty}_w(\mathbb{R}^{n+1}_+)$. The example in Proposition \ref{p:ex} shows that the image of $M$ can (in the case of slowly growing weights) be a proper subset of $h^{\infty}_w(\mathbb{R}^{n+1}_+)\bigcap\mathcal{B}(\mathbb{R}^{n+1}_+)$. For more information about the multipliers on the growth spaces see \cite{EM:14}.
\section{Concluding remark}\label{s:cr}
Consider the function $\Phi\in L^1(\mathbb{R}^n)$, and let $\Phi_y(t) = \frac{1}{y^n}\Phi\left(\frac{t}{y}\right),\; t\in\mathbb{R}^n,\, y>0$. Assume now that $\Phi\in C^{k_0}(\mathbb{R}^n)$ for some $k_0\in\mathbb{N}$, and that $\Phi^{(k)}\in L^1(\mathbb{R}^n)$ for $k\leq k_0$. Denote by $h^{\infty}_{w,\Phi}$ the space of functions $u:\mathbb{R}^{n+1}_+$ of the form $u(x,y) = \left(f*\Phi_y\right)(x)$, such that
\begin{equation}\notag
|u(x,y)| \leq Kw(y),\quad x\in\mathbb{R}^n,
\end{equation}
where $f$ is some distribution on $\mathbb{R}^n$ (we assume that this convolution exists).
The ideas and methods from \cite{EML:13} can be used to prove the following statement
\begin{theoremm}
Let $u$ be a function in $h^{\infty}_{w,\Phi}(\mathbb{R}^n_+)$. Put
\begin{equation}\notag
I_{\Phi}(x,\delta) = \int_{\delta}^1u(x,y)\,d\left(\frac{1}{w(y)}\right).
\end{equation}
If $w(y)y^{k_0-1}$ is bounded for all $y>0$, then the following inequality holds
\begin{equation}\label{e:phikern}
\limsup_{\delta\rightarrow0}\frac{I_{\Phi}(x,\delta)}{\sqrt{\log w(\delta)\log\log\log w(\delta)}} \leq C\|u\|_{w,\infty},\quad a.e.\, x\in\mathbb{R}^n
\end{equation}
for some absolute constant $C>0$.
\end{theoremm}
For $0<\alpha\leq 1$ let
\begin{equation}\notag
\Lip_{\alpha}(\mathbb{R}) = \{f:\mathbb{R}\mapsto\mathbb{R}: |f(x)-f(y)| \leq C|x-y|^{\alpha},\; x,y\in\mathbb{R}\},
\end{equation}
where $C$ is some constant that depends only on $f$, and denote by $\|f\|_{\alpha}$ the infimum of such constants.
We state the result (second part of Theorem \textbf{1}) from \cite{LN:13}
\begin{theoremm}
Fix $0<\alpha<1$. For $f \in \Lip_{\alpha}(\mathbb{R})$ and $0<\varepsilon<\frac12$ let
\begin{equation}\notag
\Theta_{\varepsilon}(f)(x) = \int_{\varepsilon}^1\frac{f(x+h)-f(x-h)}{h^{\alpha}}\frac{dh}{h},\quad x\in\mathbb{R}.
\end{equation}
Then there exists a constant $C(\alpha)$, independent of $f$, such that at almost every point $x\in\mathbb{R}$, one has
\begin{equation}\label{e:skern}
\limsup_{\varepsilon\rightarrow0}\frac{|\Theta_{\varepsilon}(f)(x)|}{\sqrt{\log \frac{1}{\varepsilon}\log\log\log \frac{1}{\varepsilon}}} \leq C(\alpha)\|f\|_{\alpha}.
\end{equation}
\end{theoremm}
These two results are actually very similar. Indeed, let $S(t) = \frac12\chi_{[-1,1]}$ be the box kernel, $w(h) = h^{\alpha-1},\; h>0,$ and $u(x,y) = (f'*S_y)(x),\; x\in\mathbb{R},\, y>0$ (here $f'$ is understood in the sense of distributions, so that the convolution is well defined). We see that the condition $f\in \Lip_{\alpha}$ is equivalent to $u \in h^{\infty}_{w,S}$. Then $\Theta_{\varepsilon}$ can be rewritten in the following way
\begin{multline*}
\frac{(1-\alpha)}{2}\Theta_{\varepsilon}(f)(x) = \int_{\varepsilon}^1\frac{f(x+h)-f(x-h)}{2h}\,(1-\alpha)h^{-\alpha}dh\\ =\int_{\varepsilon}^1\left(f'*S_{h}\right)(x)\,d\left(\frac{1}{w(h)}\right)\\=
\int_{\varepsilon}^1u(x,h)\,d\left(\frac{1}{w(h)}\right) = I_S(x,\varepsilon),\quad x\in\mathbb{R},\;\varepsilon >0.
\end{multline*}
If the box kernel $S$ were smooth enough, the inequality \eqref{e:skern} would follow immediately from \eqref{e:phikern}. Unfortunately that is not the case, and  we see that Theorem \textbf{C} cannot be deduced from Theorem \textbf{B}. To prove Theorem \textbf{C} Llorente and Nicolau used a different argument, which employed the fact that $f \in \Lip_{\alpha}$.\par
\textit{Acknowledgements}. We would like to thank A. Nicolau and E. Malinnikova for fruitful discussions of the problem. Part of the work was done while the author was visiting the Chebyshev laboratory at St. Petersburg State University. We are grateful to the laboratory for hospitality and great working conditions. 

\end{document}